\documentclass[12pt,reqno]{amsart}
\usepackage[a4paper, total={6in, 9in}]{geometry}
\usepackage{amsthm}
\usepackage{amsmath}
\usepackage{etoolbox}
\usepackage{amsfonts}
\usepackage[shortlabels]{enumitem}
\usepackage{amssymb}
\usepackage{xcolor}
\usepackage[all,cmtip]{xy}

\newtheorem{thm}{Theorem}[section]

\newtheorem{cor}[thm]{Corollary}

\newtheorem{remark}[thm]{Remark}

\newtheorem{ex}[thm]{Example}

\newcommand{\norm}[1]{\left\|#1\right\|}

\newcommand{\bb}[1]{\mathbb{#1}}
\newcommand{\cl}[1]{\mathcal{#1}}

\newcounter{egcounter}
\begin{document}

\title[A KS theorem in several variables]{A several variables Kowalski-S\l odkowski theorem for topological spaces}
\author{Jaikishan}
\address{Department Of Mathematics\\
      School of Natural Sciences\\
      Shiv Nadar Institution of Eminence\\
  Gautam Buddha Nagar\\
         Uttar Pradesh 201314, India}
\email{jk301@snu.edu.in}

\author{Sneh Lata}
\address{Department Of Mathematics\\
       School of Natural Sciences\\
      Shiv Nadar Institution of Eminence\\
  Gautam Buddha Nagar\\
         Uttar Pradesh 201314, India}
\email{sneh.lata@snu.edu.in}

\author{ Dinesh Singh}
\address{Centre For Digital Sciences\\
      O. P. Jindal Global University\\
   Sonipat\\
        Haryana 131001, India}
\email{dineshsingh1@gmail.com}

\subjclass[2020]{Primary 32A65; Secondary 32A35, 46E40, 47B38}

\keywords{Gleason-Kahane-$\dot{\rm Z}$elazko theorem, Kowalski and S\l odkowski theorem, multiplicative linear functional, Hardy spaces, multiplicative Gleason-Kahane-$\dot{\rm Z}$elazko theorem.}

\begin{abstract}
In this paper, we provide a version of the classical result of Kowalski and S\l odkowski that generalizes the famous Gleason-Kahane-$\dot{\rm Z}$elazko (GKZ) theorem by characterizing multiplicative linear functionals amongst all complex-valued functions on a Banach algebra. We first characterize maps on $\mathcal{A}$-valued polynomials of several variables that satisfy some conditions, motivated by the result of Kowalski and S\l odkowski, as a composition of a multiplicative linear functional on $\mathcal{A}$ and a point evaluation on the polynomials, where $\mathcal{A}$ is a complex Banach algebra with identity. We then apply it to prove an analogue of Kowalski and S\l odkowski's result on topological spaces of vector-valued functions of several variables. These results extend our previous work from \cite{jaikishan2024multiplicativity}; however, the techniques used differ from those used in \cite{jaikishan2024multiplicativity}. Furthermore, we characterize weighted composition operators between Hardy spaces over the polydisc amongst the continuous functions between them. Additionally, we register a partial but noteworthy success toward a multiplicative GKZ theorem for Hardy spaces. 
\end{abstract}

\maketitle
 
\section{Introduction}\label{intro}
Let $\cl A$ denote a complex Banach algebra with identity and $G(\mathcal{A})$  denote the set of invertible elements in $\mathcal{A}$. A non-zero linear functional $F$ defined on $\mathcal{A}$ is called multiplicative or a character if $F(xy)=F(x)F(y)$ for all $x,y \in \mathcal{A}$. It is easy to see that a multiplicative linear functional on $\mathcal{A}$ can never vanish on $G(\mathcal{A})$. Interestingly, the converse is also true, and it is the well-known Gleason-Kahane-$\dot{\rm Z}$elazko (GKZ)  theorem, which characterizes all multiplicative linear functionals on a complex unital Banach algebra.

\begin{thm}[GKZ theorem]\label{GKZ}
Let $F$ be a linear functional on a complex unital Banach algebra $\mathcal{A}$ with identity $e$ such that $F(e)=1$ and $F(x)\ne0$ for every $x\in G(\mathcal{A})$, then $F(xy)=F(x)F(y)$ for all $x,y\in\mathcal{A}$.
\end{thm}
	
The GKZ theorem was first proved by Gleason \cite{gleason1967characterization} and, independently, by Kahane and $\dot{\rm Z}$elazko \cite{kahane1968characterization} for commutative Banach algebras. Afterward, $\dot{\rm Z}$elazko \cite{zelazko1968characterization} extended it to the non-commutative case. An elementary proof of Theorem \ref{GKZ}, due to Roitman and Sternfeld, can be found in \cite{roitman1981linear}.

Around a decade later, Kowalski and S\l odkowski gave an interesting generalization of the GKZ theorem in \cite{kowalski1977characterization}. They proved the following characterization of multiplicative linear functionals amongst all complex-valued functions on complex unital Banach algebras without using the linearity condition.  

 \begin{thm}[KS theorem]\label{KLS}
 Let $\mathcal{A}$ be a complex unital Banach algebra with identity $e$ and let $F:\mathcal{A} \to \mathbb{C}$ be a map satisfying $F(0)=0$ and 
 \begin{equation}\label{KS}
 \big(F(x)-F(y)\big)e-(x-y)\notin G(\mathcal{A})
 \end{equation}
 for all $x,y\in\mathcal{A}$. Then $F$ is multiplicative and linear.
 \end{thm}
 
 Note that for $y=0$, Condition (\ref{KS}) yields $F(x)e-x \notin G(\mathcal{A})$ for every $x$ in $\mathcal{A}$, which for a linear $F$, is equivalent to $F$ being non-zero on 
 invertible elements in $\mathcal{A}$. Thus, the KS theorem generalizes the GKZ theorem. Interestingly, the above characterizations have analogues for function spaces that are not necessarily algebras. 

 To lend perspective to the discussion at hand, which has set the tone for our current work, we shall first discuss analogues of the GKZ theorem as proved in \cite{jaikishan2024multiplicativity, mashreghi2015gleason, sampat2021cyclicity}. Unless specified otherwise, we consider all vector spaces over the field of complex numbers and assume all the functions to be complex-valued.
 
Recently, Mashreghi and Ransford \cite{mashreghi2015gleason} extended the GKZ theorem to certain Banach spaces of analytic functions on the open unit disc $\mathbb{D}$. The following is their version for the Hardy spaces.
	
\begin{thm}\cite[Theorem 2.1]{mashreghi2015gleason}\label{MR1} For $0<p\le \infty$, let $F: H^{p}(\mathbb{D}) \rightarrow \mathbb{C}$ be a linear functional such that $F(1)=1$ and $F(g)\ne0$ for all $outer$ functions $g$ in $H^{p}(\bb D)$. Then there exists $w\in \mathbb{D}$ such that 
\begin{equation*}
F(f)=f(w) \ \ for \ all \ f\in H^{p}(\bb D).
\end{equation*}
\end{thm} 
 
This is a genuine counterpart of the GKZ theorem for Hardy spaces, as any function $f$ in the Hardy space with an $g$ in the same space such that $fg=1$ must be outer. Moreover, a point evaluation 
is multiplicative whenever the multiplication is well-defined. In the same paper, they also gave the following analogous result for a much broader class of Banach spaces. 

\vspace{.2 cm}

Let $Hol(\mathbb{D})$ denote the set of all holomorphic functions on the open unit disc $\mathbb{D}$, and 
let $\cl X\subseteq Hol(\mathbb{D})$ be a Banach space with the following properties:

\vspace{.2 cm}

\begin{enumerate}[$(\cl X1)$,nosep]
\item for each $w\in \bb D$, the map $f \mapsto f(w)$ is continuous;
\item $\cl X$ contains all polynomials, and they are dense in $\cl X$;
\item multiplication with the function $z$ is well-defined on $\cl X$.
\end{enumerate}

\vspace{.2 cm}

Further, suppose $\cl W$ is a subset of $\cl X$ which satisfies the following properties:  

\vspace{.2 cm}

\begin{enumerate}[$(\cl W1)$,nosep]
\item if $f\in \cl X$ is bounded and $|f|$ is bounded away from zero on $\mathbb{D}$, then $f\in \cl W$;
\item $\cl W$ contains the linear polynomials $g(z)=z-\lambda, \ |\lambda|=1$.
\end{enumerate}

\begin{thm}\label{MR2}\cite[Theorem 3.1]{mashreghi2015gleason}
Let $\cl X\subseteq Hol(\mathbb{D})$ be a Banach space and let $\cl W$ be a subset of $\cl X$. Suppose they satisfy the above-stated properties $(\cl X1)-(\cl X3)$ and $(\cl W1)-(\cl W2)$, respectively. If $F: \cl X \to \mathbb{C}$ is a continuous linear functional such that $F(1)=1$ and it never vanishes on the set $\cl W$, then there exists $w\in\bb D$ such that $F(f)=f(w)$ for all $f\in \cl X$.
\end{thm}

Inspired by the work mentioned above of Mashreghi and Ransford, the authors of this article proved a new analogue of the GKZ theorem in \cite{jaikishan2024multiplicativity}. This analogue is interesting because we dropped the condition that requires the vector space to be a Banach space. In fact, we assumed the vector space has a topology that is not necessarily related to its algebraic structure. More precisely, we proved the following:
 
 Let $\mathcal{F}(r,\mathbb{C})$ denote the family of complex-valued functions defined on the open disc $D(0,r)$ in $\mathbb{C}$ with center at $0$  and radius $r$.
 
\begin{thm}\label{JLS}\cite[Theorem B]{jaikishan2024multiplicativity} Let $\cl X\subset \mathcal{F}(r,\mathbb{C})$ be a complex vector space equipped with a topology that satisfies the following properties:
\begin{enumerate}[nosep]
\item for each $w\in D(0,r)$, the map $f \mapsto f(w)$ is continuous;
\item $\cl X$ contains the set of complex polynomials as a dense subset. 
\end{enumerate}
Let $\cl Y=\{(z-\lambda)^n: n\ge 0, \ \lambda\in \bb C, \  |\lambda|\ge r\}$. If $F: \cl X \to \mathbb{C}$ is a continuous linear functional such that $F(1)=1$ and $F(g)\ne0$ for all $g\in \cl Y$, then there exists $w \in D(0,r)$ such that 
\begin{equation*}
F(f)=f(w)
\end{equation*}
for all $f\in \cl X$.
\end{thm}

 Notice that Theorem \ref{JLS} is, in a spirit, similar to Theorems \ref{MR1} and \ref{MR2}. The point to be also noted about Theorem \ref{JLS}, in addition to what has been stated above, is that the hypotheses involve conditions on only a small and special class of outer functions, namely, powers of linear outer polynomials. Furthermore, Theorem \ref{JLS} gave a version of Theorem \ref{MR1}, but with a restriction of continuity on the linear functional, which has been justified in \cite{jaikishan2024multiplicativity} with an example. It is also worth noting that unlike the proofs of Theorems \ref{MR1} and \ref{MR2}, as given in \cite{mashreghi2015gleason}, the proof of Theorem \ref{JLS} in \cite{jaikishan2024multiplicativity} does not rely on the GKZ theorem.

 In \cite{sampat2021cyclicity}, Sampat obtained a version of the GKZ theorem for function spaces in the several variables setting. This result has its roots in \cite{kou2017linear} and \cite{mashreghi2015gleason}. In \cite{kou2017linear}, Kou and Liu characterized the bounded linear operators on $H^p(\bb D), \ 1<p<\infty$ that preserve outer functions as weighted composition operators. In this process, Kou and Liu \cite{kou2017linear} gave yet another proof of Theorem \ref{MR1} in which they worked with a specific subclass of outer functions, namely, $\{e^{wz}: w\in\mathbb{C}\}$. Furthermore, 
 their proof does not rely on the GKZ theorem, and the result assumes the continuity of the linear functional. Sampat extended Kou and Liu's ideas to the several variables setting in \cite{sampat2021cyclicity}. The following is his analogue of the GKZ theorem. For an open set $U \subset \mathbb{C}^n$, let $\cl X \subset Hol(U)$ be a Banach space satisfying the following properties:

\vspace{.2 cm}

\begin{enumerate}[(JS1), nosep]
\item The set of polynomials is dense in $\cl X$;
\item For each $z\in U,$ the point evaluation map $f\mapsto f(z)$ is bounded on $\cl X$. Furthermore, $U$ contains each $z \in \bb C^n$ for which the map $p\mapsto p(z)$ on the set of polynomials extends to a bounded linear functional on all of $\cl X$;
\item The $i^{th}$- shift operator $S_i: \cl X \to \cl X$, defined as $S_i f(z):=z_if(z)$ for every $(z_k)_{k=1}^n=z\in U$ and $f\in \cl X$, is a bounded linear operator for all $1\le i\le n$.
\end{enumerate}

\begin{thm}\label{Sam}\cite[Theorem 1.4]{sampat2021cyclicity}
Let $\cl X$ be a Banach space that satisfies the above properties $(JS1)-(JS3)$ over an open set $U\subset \mathbb{C}^n$. Let $\Lambda : \cl X \to \mathbb{C}$ be a continuous linear functional such that $\Lambda(e^{w\cdot z})\neq 0$ for every $w\in \mathbb{C}^n$. Then there exist a non-zero $ \alpha \in \mathbb{C}$  and $w\in U$ such that
\begin{equation*}
\Lambda(f)=\alpha f(w) \quad \text{for all } f\in \cl X.
\end{equation*}
\end{thm}
	 
We now focus on an analogue of the KS theorem for function spaces. By following the ideas of Mashreghi and Ransford from \cite{mashreghi2015gleason}, Sebastian and Daniel extended the KS theorem to modules in \cite[Theorem 2.3]{sebastian2021weaker}; then they used it to obtain the following generalization of Theorem \ref{MR1}.

\begin{thm}\label{SD-h}\cite[Theorem 3.1]{sebastian2021weaker}
Let $1\le p\le \infty$ and $F: H^p(\bb D) \rightarrow \mathbb{C}$ be a function such that $F(0)=0, \ F(1)\ne 0,$ and 
\begin{equation}
\big(F(f_1)-F(f_2)\big)1-F(1)(f_1-f_2)\notin S
\end{equation}
for all $f_1, \ f_2\in H^p(\bb D),$ where $S$ is the set of all outer functions in $H^p(\bb D)$. Then there exists a unique character $\chi$ on $H^\infty(\bb D)$ such that 
\begin{equation*}
F(fg)=\chi(f)F(g)
\end{equation*}	
for all $f\in H^\infty(\bb D)$ and $g\in H^p(\bb D).$ Further, there exist a non-zero constant $c$ and $w\in \bb D$ such that $F(g)=cg(w)$ for all $g\in H^\infty(\bb D).$		
\end{thm}

Apart from being captivating in their own right, these various analogues of the GKZ theorem and its generalization, the KS theorem, are attracting attention as they lead to, in their respective settings, characterizations of linear maps that preserve outer functions. Mashreghi and Ransford used 
Theorem \ref{MR1} to show that linear operators on Hardy spaces $H^p(\mathbb{D}), \ 0<p\le \infty$ preserving outer functions are necessarily weighted composition operators (\cite[Theorems 2.2 \& 2.3]{mashreghi2015gleason}). Sebastian and Daniel in \cite{sebastian2021weaker}, Kou and Liu in \cite{kou2017linear}, and Sampat in \cite{sampat2021cyclicity} have also used their results 
mentioned above to obtain similar results for their respective settings.   

In this paper, we extend the KS theorem to a much larger class of topological spaces (Theorem \ref{2.3})—as opposed to Banach spaces—consisting of functions defined on a bounded subset of $\bb C^n$ and taking values in a fixed complex Banach algebra with identity. To be precise, given a topological space $\cl X$ of functions defined on certain bounded subset of $\mathbb C^n$ taking values in a complex Banach algebra $\mathcal{A}$ with identity, we characterize non-zero continuous functions on $\cl X$ that satisfy some conditions similar to those imposed in the KS theorem as a composition of a point evaluation on $\cl X$ and a multiplicative linear functional on $\mathcal{A}$. We note that Theorem \ref{2.3} extends our earlier work Theorem \ref{JLS} to functions of several variables; nevertheless, the techniques used in this paper differ from those employed for the proof of Theorem \ref{JLS} in \cite{jaikishan2024multiplicativity}. In addition, it is worth mentioning that many interesting and well-known function spaces such as the $\bb C^m$-valued Hardy spaces on $\mathbb D^n$, the Drury-Arveson space, the Dirichlet-type spaces on $\mathbb D^n$, and the polydisc algebra are covered by our result. A more elaborated list is given in Section \ref{GKZ-top}.  

Furthermore, we use Theorem \ref{2.3} to characterize weighted composition operators between Hardy spaces (Theorem \ref{w-comp}) amongst the family of continuous functions between them. Lastly, in Section \ref{converse}, we give a partial but noteworthy result towards a multiplicative GKZ theorem for Hardy spaces. 

 As we mentioned earlier, \cite{sebastian2021weaker} has extended the KS theorem to Hardy spaces 
(Theorem \ref{SD-h}), and \cite{sampat2021cyclicity} has given a several variables version 
of the GKZ theorem (Theorem \ref{Sam}). So, listing some aspects where our results differ from those in \cite{sampat2021cyclicity} and \cite{sebastian2021weaker} is natural and necessary.

\begin{enumerate}[(i)]
\item Theorem \ref{Sam} (\cite[Theorem 1.4]{sampat2021cyclicity}) is a several variables version of the GKZ 
theorem, whereas our result (Theorem \ref{2.3}) is a several variables version of the KS theorem. In addition, our 
functions are vector-valued as well.

\item The author in \cite{sampat2021cyclicity} uses some technical complex function theory results,  whereas our arguments are elementary.

\item We assume the given map to be non-zero on certain polynomials, whereas Theorem \ref{SD-h}  assumes the map to be non-zero on the set of all outer functions, and Theorem \ref{Sam} assumes it to be non-zero on the set $\{e^{w\cdot z}:w\in \mathbb{C}\}.$ 

\item Theorem \ref{SD-h} extends the KS theorem only to Hardy spaces $H^p(\mathbb{D}), \ 1\le p\le \infty,$ whereas Theorem \ref{2.3} covers a much broader class including the vector-valued Hardy spaces $H^p(\bb D^n, \bb C^m), \ 0<p\le \infty$, the Drury-Arveson space, the Dirichlet-type spaces over the polydisc $\bb D^n$, the polydisc algebra.

\item The arguments deployed by us to prove Theorem \ref{2.3} are much simpler compared to ones used by the authors of \cite{sebastian2021weaker} to prove Theorem \ref{SD-h}. The arguments used in  \cite{sebastian2021weaker} are specific to properties of Hardy spaces and can't be used for our setting. 

\item In comparison to Theorem \ref{SD-h}, Theorem \ref{2.3} has an extra hypothesis of continuity, which we show (Example \ref{cont-nec}) to be essential to deduce the desired conclusion for our setting. We also note that Theorem \ref{Sam} does include continuity in its hypotheses, but its necessity is not addressed in \cite{sampat2021cyclicity}. 

\item Unlike Theorem \ref{Sam} and \ref{SD-h}, Theorem \ref{2.3} is not restricted to spaces of analytic functions only.  
\end{enumerate}

In addition, in Theorem \ref{SD-improved}, we give an improvement of Theorem \ref{SD-h}. Indeed, the conclusion of Theorem \ref{SD-h}, when compared to Theorem \ref{MR1}, is not up to satisfaction as it does not assert the behaviour of the map on unbounded functions in $H^p(\bb D).$ Interestingly, we show in Theorem 
\ref{SD-improved} that the hypotheses of Theorem \ref{SD-h} are enough to establish that the assertion of Theorem \ref{SD-h} holds for the entire $H^p(\mathbb D)$, which is in parallel to Theorem \ref{MR1}. 

 \section{Kowalski-S\l odkowski's Theorem for Hardy spaces}
In this section, we shall refine Theorem \ref{SD-h}. We shall show that the hypotheses of Theorem \ref{SD-h} are indeed enough to establish that the map, as desired, is a point evaluation (up to a constant). We want 
to note that we fail to understand why the authors have restricted the result to only $p\ge 1$ when their 
arguments are valid for $0<p\le \infty.$ The following is our rectification of Theorem \ref{SD-h}. 

\begin{thm}\label{SD-improved}
Let $0< p\le \infty$ and let $F: H^p(\bb D) \rightarrow \mathbb{C}$ be a non-zero function such that $F(0)=0$ and 
\begin{equation}
\big(F(f_1)-F(f_2)\big)1-F(1)(f_1-f_2)\notin S
\end{equation}
for all $f_1, \ f_2\in H^p(\bb D),$ where $S$ is the set of all outer functions in $H^p(\bb D)$. Then there exist a non-zero constant $c$ and $w\in \bb D$ such that 
$$
F(f)=cf(w)
$$		
for every $f\in H^p(\bb D)$.
\end{thm}
\begin{proof} The proof of Theorem \ref{SD-h} as given in \cite{sebastian2021weaker} uses \cite[Theorem 2.3]{sebastian2021weaker}. We note that even for $0<p<1,$ the triple, $\cl A=H^\infty(\bb D), \ \cl M=H^p(\bb D), \ S$ equals the set of all outer functions in $H^p(\bb D)$ satisfy the hypotheses of Theorem 2.3 from \cite{sebastian2021weaker}, and so just as shown in \cite{sebastian2021weaker}, there exists a character $\chi:H^\infty(\bb D)\to\bb C$ such that
$F(fg)=\chi (f)F(g)$ for all $f\in H^\infty(\bb D)$ and $g\in H^p(\bb D)$, where $0<p\le \infty$. Further, by using arguments as were used in Theorem \ref{SD-h}, there exists a point $w\in \bb D$ such that
$$
\chi(f)=f(w)
$$
for all $f\in H^\infty(\bb D).$
This implies that for every $f\in H^\infty(\bb D)$, we obtain $F(f)=\chi(f)F(1)=F(1)f(w).$  

Let $O$ be an outer function in $H^p(\bb D).$ Then there exist outer functions $O_1$ and $O_2$ in $H^\infty(\bb D)$ such that $O=O_1/O_2.$ This means $O_1=O_2O$, and so $F(O_1)=\chi(O_2)F(O).$ Therefore, $F(O)=F(1)O_1(w)/O_2(w)=F(1)O(w).$ This, along with the inner-outer factorization of functions in $H^p(\bb D)$, implies that $F(f)=F(1)f(w)$ for every $f\in H^p(\bb D).$  
\end{proof}

Now, using Theorem \ref{SD-improved} and employing arguments similar to the ones used in  \cite[Theorem 2.2]{mashreghi2015gleason}, we can improve Theorem 3.2 from \cite{sebastian2021weaker} by extending it from $H^\infty(\bb D)$ to $H^p(\bb D)$ as follows.

\begin{thm}\label{SD-pre1} For $0<p\le \infty$, let $T:H^p(\bb D) \to Hol(\bb D)$ be a continuous map satisfying $T(0)=0, \ T(1)(z)\ne 0$ for every $z\in \bb D,$ and 
$$
(Tf(z)-Tg(z))1-T1(z)(f-g)\notin S
$$ 
for every $f, g\in H^p(\bb D)$ and $z\in \bb D,$ where $S$ is the set of outer functions in $H^p(\bb D).$
Then there exist holomorphic functions $\phi:\bb D\to \bb D$ and $\psi:\bb D\to \bb C\setminus \{0\}$ such that 
$$
Tf=\psi \cdot (f\circ \phi).
$$
for every $f\in H^p(\bb D).$
\end{thm}
\begin{proof} For each fixed $w\in \bb D$, the map $T_w:H^p(\bb D)\to \bb C$ given by $T_w(f)=(Tf)(w)$ satisfies the hypotheses of Theorem \ref{SD-improved}; therefore, there exist a point $a_w\in \bb D$ and a constant $c_w$ such that $(Tf)(w)=c_w f(a_w)$ for every $f\in H^p(\bb D).$ Then by taking $f$ to be the constant function $1$, we see that $c_w=T1(w)\ne 0.$ Further, by taking $f(z)=z$, we see 
$a_w=(Tz)(w)/T1(w).$ We now define $\psi=T1$ and $\phi=(Tz)/\psi$. Then $\psi$ and $\phi$ are holomorphic function on $\bb D$ such that $\psi$ doesn't vanish at any point in $\bb D$, 
$\phi(\bb D)\subseteq\bb D$, and $Tf=\psi \cdot (f\circ \phi)$. This completes the proof. 
\end{proof}

\section{Kowalski-S\l odkowski's theorem for topological spaces}\label{GKZ-top}
In this section, we shall prove an analogue of the KS theorem (Theorem \ref{KLS}) for topological spaces of functions of several variables. 

Let $\mathbb{Z}^{+}(n)$ denote the set of $n$-tuples $\alpha=(\alpha_1, \dots, \alpha_n)$ of 
non-negative integers. For $\alpha\in \bb Z^{+}(n)$ and $z\in \bb C^n,$ define $z^{\alpha}:=z_1^{\alpha_1}z_2^{\alpha_2}\dotsm z_n^{\alpha_n}$. For a Banach algebra $\cl A$ with identity, 
let $\cl P^\mathcal{A}_n$ 
denote the set of 
all $\cl A$-valued polynomials in $n$ variables, that is,  
\begin{equation*}
\mathcal{P}^{\mathcal{A}}_n=\left\{p(z)=\sum\limits_{|\alpha|\le k}a_{\alpha}z^{\alpha}:  a_{\alpha}\in \mathcal{A} , \ z\in \mathbb{C}^n \text{ and $k$ is a non-negative integer}\right \},
\end{equation*}
where $|\alpha|=\alpha_1+\alpha_2+\dotsm + \alpha_n$. 
For convenience, we shall write $\cl P_n$ instead of $\cl P_n^{\cl A}$ if $\cl A=\bb C.$ Clearly, given any compact subset $K$ of $\bb C^n$, $\mathcal{P}^\mathcal{A}_n$ can be viewed as a subalgebra of the Banach algebra $C(K,\mathcal{A})$, the set of all continuous $\cl A$-valued functions with the 
sup norm. We say $z_0$ is zero of a polynomial $p\in \cl P_n^{\cl A}$ if $p(z_0)=0$. For the rest of this section, let 
$\cl A$ represents a fixed complex Banach algebra with identity $e$. 

The following result is integral to our version of the KS theorem for topological spaces (Theorem \ref{2.3}). We present it as an independent result as, apart from being interesting in its own right, it brings out the crux of the method used to prove Theorem \ref{2.3} even more prominently. It characterizes multiplicative linear functionals on $\mathcal{P}_n^\mathcal{A}$ as non-zero complex-valued functions on $\mathcal{P}_n^\mathcal{A}$ that satisfy a condition very similar to the one assumed in the KS theorem. 
First, we recall the notion of a polynomial convex hull of a set in $\bb C^n.$ The polynomial convex hull $\widehat{D}$ of a bounded subset $D$ of $\bb C^n$ is consisting of all $w\in \bb C^n$ such that 
$$
|p(w)|\le \sup_{z\in D}|p(z)|
$$
for all polynomials $p\in \cl P_n$. Note that $\widehat{D}$ is a closed set and $D\subseteq \widehat{D}.$ 

\begin{thm}\label{Result1} Let $\Lambda : \mathcal{P}_n^\mathcal{A} \to \mathbb{C}$ be a non-zero function with $\Lambda(0)=0$. Suppose $D$ is a bounded subset of $\bb C^n$ such that for every $p$ and $q$ in $\cl P^{\cl A}_n$ the polynomial   
\begin{equation*}
\big(\Lambda(p)-\Lambda(q)\big)e-(p-q) 
\end{equation*}
has a zero in $D$. Then there exist $w\in \widehat{D}$ and a multiplicative linear functional $\phi$ on $\cl A$ such that
\begin{equation*}
\Lambda(p)= \phi(p(w)) \quad \text{for all } p\in \mathcal{P}_n^\mathcal{A}.
\end{equation*}
Moreover, we assert $w\in D$ whenever 
\begin{enumerate}[(i)]
\item $n=1$;
\item $D=D_1\times \cdots \times D_n$ for some bounded sets $D_1, \dots, D_n$ in $\bb C;$
\item $D$ is an open or a closed Euclidean ball in $\bb C^n$. 
\end{enumerate}
\end{thm}

\begin{proof}
For every $p, \ q\in \cl P_n^\mathcal{A}$ there exists  $w_{p,q}\in D$ such that $(\Lambda(p)-\Lambda(q))e=p(w_{p,q})-q(w_{p,q})$. 
In particular, for each polynomial $p$, there is a point $w_p$ in $D$ such that $\Lambda(p)e=p(w_p)$. This gives us 
	\begin{equation}\label{cont}
		\norm{(\Lambda(p)-\Lambda(q))e}=\norm{p(w_{p,q})-q(w_{p,q})}\le\sup_{z\in \overline{D}}\norm{p(z)-q(z)}.
	\end{equation}
Let $\overline{\mathcal{P}_n^\mathcal{A}}$ denote the closure of $\mathcal{P}_n^{\cl A}$ 
with respect to the sup norm on $C(\overline{D}, \cl A)$, and let $f$ be an element of $\overline{\mathcal{P}_n^\mathcal{A}}$. 
Then there is a sequence $\{p_n\}_n$ of polynomials that 
converge to $f$, which, using Inequality (\ref{cont}), yields the existence of 
$\lim \limits_{n\to\infty} \Lambda(p_n)$. This allows us to  define a function $\Lambda_1: \overline{\mathcal{P}^\mathcal{A}_n} \to \mathbb{C}$ as  $\Lambda_1(f):= \lim\limits_{n\to\infty}\Lambda(p_n)$, 
where $p_n \in \mathcal{P}_n^{\cl A}$ such that $p_n \longrightarrow f$. 
It is easy to verify that $\Lambda_1$ is a well-defined function. 
Indeed, if there is another sequence $\{q_n\}_n$ converging to $f$, then from Inequality (\ref{cont}), we get 
$\norm{\big(\Lambda(p_n)-\Lambda(q_n)\big)e}\le \norm{p_n-q_n}_\infty$, from where it follows that $\lim\limits_{n\to\infty} \Lambda(p_n)=\lim \limits_{n\to\infty}\Lambda(q_n)$.
	
Now we want to show that $\big(\Lambda_1(f)-\Lambda_1(g)\big)e-(f-g) \notin G(\overline{\mathcal{P}_n^\mathcal{A}})$ (the set of invertible elements of $\overline{\mathcal{P}_n^\mathcal{A}}$) for all $f,g\in \overline{\mathcal{P}_n^\mathcal{A}}$. 
Suppose there are $f$ and $g$ in $\overline{\mathcal{P}_n^\mathcal{A}}$ such that $\big(\Lambda_1(f)-\Lambda_1(g)\big)e-(f-g)\in G(\overline{\mathcal{P}_n^\mathcal{A}})$. Let $\{p_n\}_n$ and 
$\{q_n\}_n$ be sequences of polynomials in $\cl P_n^{\cl A}$ such that $p_n\longrightarrow f$ and $q_n\longrightarrow g$ in $\overline{\cl P_n^\cl A}$. Then  
\begin{equation*}
\big(\Lambda_1(p_n)-\Lambda_1(q_n)\big)e-(p_n-q_n) \longrightarrow \big(\Lambda_1(f)-\Lambda_1(g)\big)e-(f-g)
\end{equation*}
in $\overline{\cl P_n^{\cl A}}$ as $n\longrightarrow \infty$. Since $G(\overline{\mathcal{P}_n^\mathcal{A}})$ is open,  there 
exists a positive integer $N$ such that 
$\big(\Lambda_1(p_n)-\Lambda_1(q_n)\big)e-(p_n-q_n)\in G(\overline{\mathcal{P}_n^\mathcal{A}})$ for all $n\ge N$. In particular, $\big(\Lambda_1(p_N)-\Lambda_1(q_N)\big)\mathbf{e}-(p_N-q_N)$ can't vanish on $D$. But this is a contradiction to the hypothesis. So, $(\Lambda_1(f)-\Lambda_1(g))e-(f-g)\notin G(\overline{\mathcal{P}_n^\mathcal{A}})$ for all $f,g$ in $\overline{\mathcal{P}_n^\mathcal{A}}$.  

Then, by Theorem \ref{KLS}, $\Lambda_1$ is multiplicative and linear. In particular, $\Lambda$ is multiplicative and linear on $\mathcal{P}_n^\mathcal{A}$.
	
For each $1\le i\le n,$ define $r_i(z):=z_i e$. Let $w=(\Lambda(r_1),\Lambda(r_2),\dots ,\Lambda(r_n))$. Since $\Lambda$ is multiplicative, we obtain  $\Lambda(z^{\alpha})=\Lambda(e z_1^{\alpha_1}) \Lambda(e z_2^{\alpha_2})\cdots \Lambda(e z_n^{\alpha_n})=w^{\alpha}$. Now consider the map $\phi: \mathcal{A} \to \bb C$ given by $\phi(a)=\Lambda(a)$. Then for any polynomial $p(z)=\sum\limits_{|\alpha|\le k}a_{\alpha}z^{\alpha}$ in $\mathcal{P}_n^{\cl A}$, we have
\begin{eqnarray*}
\Lambda(p)&=&\sum\limits_{|\alpha|\le k}\Lambda(a_{\alpha}z^{\alpha})=\sum\limits_{|\alpha|\le k}\Lambda(a_{\alpha})\Lambda(e z^{\alpha})\\
&=&\sum\limits_{|\alpha|\le k}\phi(a_{\alpha})\Lambda(e z^{\alpha})=\sum\limits_{|\alpha|\le k}\phi(a_{\alpha})w^{\alpha}\\
&=&\phi\left(\sum\limits_{|\alpha|\le k}a_{\alpha}w^{\alpha}\right)=\phi(p(w))
\end{eqnarray*}
This establishes $\Lambda(p)=\phi(p(w))$ for all $p\in \cl P_n^{\cl A}$.  

To prove $w\in \widehat{D},$ let $p$ be a polynomial in $\cl P_n.$ Then $p_1(z)=p(z) e$ is a polynomial in $\cl P_n^{\cl A};$ thus, 
$\Lambda(p_1)=\phi(p_1(w))=p(w).$ This, using Inequality (\ref{cont}), implies that 
$$
|p(w)|=|\Lambda(p_1)|\le \sup_{z\in \overline{D}}|p(z)|=\sup_{z\in D}|p(z)|.
$$
Hence, $w\in \widehat{D}.$  

Finally, we will establish the assertions made in the moreover part. First, suppose $n=1,$ that is, we are in the one-variable situation. Take $p(z)=e z,$ then $\Lambda(p)=w.$ Also, as we noted at the beginning of the proof, there exists a point $a\in D$ such that 
$\Lambda(p)e=p(a)=e a$, which implies that $w=a\in D.$  

We now turn to the general $n$ case. Let $D=D_1\times \cdots \times D_n$ for some subsets $D_1, \dots, D_n$ of $\bb C.$ The idea here is very similar to case $n=1$. Note that $\Lambda(r_i)=\phi(r_i(w))=\phi(w_i e)=w_i$. Additionally, as mentioned earlier, some $a$ exists in $D$ such that $\Lambda(r_i) e=r_i(a)=e a_i.$ Thus $w_i=a_i\in D_i$; hence, $w\in D_1\times \cdots \times D_n=D.$  

Lastly, let $D=\{z\in \bb C^n: ||z-b||_{\bb C^n}<r\}$ for some $b\in \bb C^n$ and positive real number $r.$ Let $p(z)=\sum_{i=1}^n (\overline{w_i-b_i})e(z_i-b_i e)$. Then 
$$
\Lambda(p)=\phi(p(w))=\sum_{i=1}^n |w_i-b_i|^2=||w-b||^2_{\bb C^n}.
$$
Also, there exists some $a\in D$ such that $\Lambda(p) e=p(a)$, which implies that 
$||w-b||^2_{\bb C^n}=\sum_{i=1}^n (\overline{w_i-b_i})(a_i-b_i)=\langle{w-b,a-b}\rangle_{\bb C^n}$. Thus, using the  
Cauchy-Schwarz inequality, we infer $||w-b||_{\bb C^n}\le ||a-b||_{\bb C^n}<r.$ Thus, $w\in D$. A similar set of arguments will work when $D$ is a closed Euclidean ball in $C^n$. This completes the proof.  
\end{proof}

We are ready to prove our version of the KS theorem for topological spaces of functions of several variables that take values in a fixed complex Banach algebra with identity. Before we state our result, we first fix some notations for it. Recall that we have fixed $\cl A$ for Banach algebra with identity $e$. Let $D$ be a bounded subset of $\bb C^n$ and $\cl F(D, \mathcal{A})$ denotes the vector space of all $\mathcal{A}$-valued functions defined on $D$. We assume that when $n\ge 2$, $D$ is either a Cartesian product of $n$ bounded subsets of $\bb C$ or an Euclidean ball in $\bb C^n$.

\begin{thm}\label{2.3} Let $\cl X$ be a subset of $\cl F(D, \mathcal{A})$ equipped with a topology that satisfies the following properties:
\begin{enumerate}[(i),nosep]
\item  for each $z\in D$, the evaluation map $f \mapsto  f(z)$ is continuous;
\item  $\cl X$ contains the polynomial set $\mathcal{P}_n^\mathcal{A}$ as a dense subset.
\end{enumerate}

\vspace{.2 cm}

\noindent Let $\Lambda :\cl X \to \mathbb{C}$ be a non-zero continuous function such that $\Lambda({0})=0$  and for every $p$ and $q$ in $\mathcal{P}_n^\mathcal{A}$ 
the polynomial  
$$
\big(\Lambda(p)-\Lambda(q)\big)e-(p-q) 
$$ 
has a zero in $D$. Then there exist $w\in D$ and a multiplicative linear functional $\phi$ on $\mathcal{A}$ such that
	\begin{equation*}
		\Lambda(f)= \phi(f(w)) \quad \text{for every } f\in \cl X.
	\end{equation*} 
\end{thm}
\begin{proof}
Clearly, the restriction $\Lambda: \mathcal{P}_n^\mathcal{A}\longrightarrow \mathbb{C}$  satisfies the hypotheses of Theorem \ref{Result1}. Then, using the moreover part of Theorem \ref{Result1}, there exist $w\in D$ and a multiplicative linear functional $\phi$ on $\mathcal{A}$ such that $\Lambda(p)= \phi(p(w))$ for every  
$p\in \cl P_n^{\cl A}$ . Now the continuity of $\Lambda$ and the denseness of $\mathcal{P}_n^{\cl A}$ in $\cl X$ yield    
	\begin{equation*}
		\Lambda(f)= \phi(f(w)) \quad \text{for all } f\in \cl X.
	\end{equation*}
\end{proof}

\begin{remark}
The hypothesis ``$ \big(\Lambda(p)-\Lambda(q)\big)e-(p-q)$ has a zero in $D$" in Theorems \ref{Result1} and \ref{2.3} for the case when 
$n=1, D=\bb D \ {\rm and} \ \cl A=\bb C$ assumes nothing but that the polynomials $\big(\Lambda(p)-\Lambda(q)\big)e-(p-q)$ does not belong to the class of polynomials that are outer functions. 
Consequently, our hypothesis is in line with the hypothesis of the analogues of the GKZ theorem and the KS theorem recorded in the Introduction.  
\end{remark}

\begin{remark} Note that Theorem \ref{2.3} generalizes our earlier work (Theorem \ref{JLS}) from \cite{jaikishan2024multiplicativity} in three ways. It extends Theorem \ref{JLS} to topological spaces that are not necessarily vector spaces. In addition, it extends Theorem \ref{JLS} to vector-valued functions of several variables. As a result, it generalizes Theorem \ref{MR1}—Mashreghi and Ransford's version of the GKZ theorem for Hardy spaces—to vector-valued functions of several variables. Furthermore, it extends Theorem \ref{SD-h}—the recent extension of the KS theorem—to vector-valued functions of several variables. It is worth noting that the methods employed in proving Theorem \ref{2.3} are entirely independent of the ones used in proving Theorem \ref{JLS}.
\end{remark}

\begin{remark}
We would like to now compare Theorem 3.2 with Theorem \ref{Sam}
(\cite[Theorem 1.4]{sampat2021cyclicity}) that also is a several variables analogue of the GKZ theorem for function spaces. First, we note that Theorem \ref{Sam} is for Banach spaces of scalar-valued functions, whereas Theorem \ref{2.3} deals with topological spaces of vector-valued functions.
Next, we compare the domains of the function spaces considered by the two theorems. Theorem 3.2, in the several variables case, is proved under the assumption that the domain is either a Euclidean ball or a Cartesian product. In contrast, it may seem that the domain of functions in Theorem \ref{Sam} includes a big class of subsets of $\bb C^n$. But, we want to point out that Theorem \ref{Sam} assumes the domain to be a ``maximal domain" (hypothesis JS2 in Theorem 1.6), and if we assume the domain $D$ in Theorem \ref{2.3} to be a maximal domain instead of being a specific set such as a Euclidean ball or a Cartesian product (as done in Theorem \ref{2.3}), then indeed with the help of Theorem \ref{Result1}, for this case also we would obtain the map $\Lambda$ to be a point evaluation for the scalar-valued case and composition of a multiplicative linear functional and a point evaluation for the vector-valued case. Moreover, Theorem \ref{Sam}
doesn't cover Dirichlet spaces $\cl D_\alpha$ for $\alpha>1$ as the maximal domain for functions in these Dirichlet
spaces is $\overline{\bb D^n}$ and not $\bb D^n$. But, these are covered by Theorem \ref{2.3} as they satisfy its hypotheses. 
\end{remark}

\noindent The following is a list of some well-known function spaces that satisfy the conditions of Theorem \ref{2.3}. 

 \begin{itemize}
 \item The vector-valued Hardy spaces $H^p(\bb D^n, \bb C^m)$ for $0<p\le \infty,$ where we consider the weak-start topology for $p=\infty;$   
\item the Drury-Arveson Space $\mathcal{H}_n ^2$; 
\item the Dirichlet-type spaces $\mathcal{D}_{\alpha}$ for $\alpha \in \mathbb{R}$;
\item the polydisc algebra $A(\mathbb{D}^n)$;
\item the ball algebra $A(\bb B_n)$;
\item the Bergman spaces $A^p_\alpha(0<p<\infty)$ defined on the open unit ball $\bb B_n$ of $\bb C^n$;
\item the little Bloch space $\mathcal{B}_0(\bb B_n)$;

\item the space VMOA of functions of vanishing mean oscillation defined on $\bb B_n$.
\end{itemize}

Although we have already compared Theorem \ref{2.3} with Theorems \ref{Sam} and \ref{SD-h}, only to bring out this comparison more prominently for the reader, we give below the particular case of Theorem \ref{2.3} for $\cl X=H^p(\mathbb{D}^n, \bb C^m).$ For $0<p\le\infty$, 
let $H^p(\bb D^n,\bb C^m)$ denote the Hardy space of $\bb C^m$-valued functions on $\bb D^n$. Recall that $\bb C^n$ with Euclidean norm is a Banach algebra with identity $\mathbf{1}=(1,\dots,1)^T$. Note that in Theorem \ref{2.3} we do not require the norm of the unit element to be 1.

\begin{cor}\label{5.3}
For $n, m\in \bb N$ and $0<p\le \infty$, let $\Lambda : H^p(\mathbb{D}^n, \mathbb{C}^m) \to \mathbb{C}$ be a non-zero continuous (weak-star continuous for $p=\infty$) function such that $\Lambda({0})=0$ and for polynomials $p$ and $q$ in $\cl P^{\bb C^{m}}_{n}$ the polynomial $\big(\Lambda(p)-\Lambda(q)\big){\bf 1}-(p-q)$ has a zero in $\bb D^n$. Then there exist $w\in \mathbb{D}^n$  and $1\le k\le m$ such that
		\begin{equation*}
			\Lambda(f)= \langle f(w),e_k \rangle_{\bb{C}^m}
		\end{equation*}
			for all $f\in H^p(\bb D^n,\bb C^m)$, where $\{e_1,\dots,e_m\}$ is the standard orthonormal basis of $\bb C^m$. 
	\end{cor}

In Theorem \ref{2.3}; hence in Corollary \ref{5.3}, we assume the function $\Lambda$ to be continuous, whereas the comparative result, Theorem \ref{SD-h}, which again is an analogue of the KS theorem for the Hardy spaces, do not assume the function to be continuous. With Example \ref{cont-nec}, we establish the necessity for it to obtain the desired conclusion. Theorem \ref{Sam} is also a version of the GKZ for Hardy spaces on polydisc that assumes continuity, but the continuity assumption is never justified in \cite{sampat2021cyclicity}, which could be because this paper is inspired by work in \cite{kou2017linear}, which in turn is motivated by a problem in the area of geophysical imaging where the continuity assumption seems natural. 

\begin{ex}\label{cont-nec}
Let $\mathcal{B}$ be a Hamel basis for $H^2(\mathbb{D})$ obtained by extending the orthonormal basis $\mathcal{B}_1=\{z^n:n \in \mathbb{N}\cup\{0\}\}.$ 
 Let $f\in H^2(\mathbb{D})$. Then there exist elements $z^{r_1}, z^{r_2}, \dots, z^{r_n}$ in $\mathcal{B}_1$ and $f_{t_1},f_{t_2},\dots, f_{t_m}$ in $\cl B\setminus \mathcal{B}_1$, and scalars $c_1,c_2,\dots c_n, d_{1}, d_2,\dots ,d_{m}$ in $\mathbb{C}$ such that 	
 
\begin{equation*}
f=\sum_{i=1}^{n}c_iz^{r_i}+\sum_{i=1}^{m}d_if_{t_i}=p+g
\end{equation*}
where $p=\sum_{i=1}^{n}c_iz^{r_i}$ and $g=\sum_{i=1}^{m}d_if_{t_i}$. For a fixed $w \in \mathbb{D}$, define $F$ at $f$ by 
\begin{equation*}
F(f)=p(w)+ \delta
\end{equation*}
where $\delta$ is a non-zero scalar. It is easy to see that $F$ is well-defined and satisfies the hypotheses of Theorem \ref{2.3}, but $F$ is not linear. This implies $F$ cannot be continuous because otherwise, given Theorem 3.2, $F$ must be a point evaluation, which forces it to be linear. 
\end{ex}

Finally, we conclude this section with the following characterization of weighted composition operators between two Hardy spaces.

\begin{thm}\label{w-comp} For $0<p, q\le \infty,$ let $T: H^p(\bb D^n) \to H^q(\bb D^m)$ be a continuous function, where we consider weak-start topology on $H^p(\bb D^n)$ when  $p=\infty.$ Suppose $T(0)=0, \ T(1)(w)\ne 0$ for each $w\in \bb D^m$. If given any pair $p, q$ of complex-valued polynomials in $n$-variables and $w\in \bb D^m$ the polynomial 
$$
\big(Tp(w)-Tq(w)\big)1-\big(T1(w)\big)(p-q)
$$
has a zero in $\bb D^n$, then there exist $\psi\in H^q(\bb D^m)$ and an analytic function $\phi:\bb D^m\to \bb D^n$ such that
$$
T(f)= \psi\cdot (f\circ \phi)
$$ 
for every $f\in H^p(\bb D^n).$ 
\end{thm}
\begin{proof}  For each fixed $w\in \bb D^m,$ define the map $T_w: H^p(\bb D^n) \to \mathbb{C}$ given by $T_w(f)=(Tf)(w)$. It is easy to see that $T_w/(T1)(w)$ satisfies all the conditions in Corollary \ref{5.3}, so there exists $\zeta\in\bb D^n$ such that $(Tf)(w)=(T1)(w)f(\zeta)$ for all $f\in H^p(\bb D^n)$. By taking $f=z_i$, the $i^{th}$ coordinate function, we get $\zeta=(Tz_1(w),\dots, Tz_n(w))/(T1)(w) $. Let $\psi=T(1)$ and define $\phi(w)=(Tz_1(w),\dots, Tz_n(w)) /\psi(w)$.  Then $\phi(w) \in\bb D^n$ and $(Tf)(w)=\psi(w)f(\phi(w))$ for every $f\in H^p(\bb D^n)$. This completes the proof.
\end{proof}

\section{A Multiplicative GKZ Theorem for Hardy spaces}\label{converse}
In this section, we shall investigate a multiplicative analogue of the GKZ theorem within the context of Hardy spaces. By a multiplicative analogue of the GKZ theorem, we mean a result that asserts the linearity of a complex-valued multiplicative function on an unital Banach algebra that assumes a value in the spectrum. The first positive result concerning multiplicative maps taking values in the spectrum was obtained by Maouche \cite{maouche1996formes}.    

\begin{thm}[Maouche]\label{Mao}
	Let $\mathcal{A}$ be a unital Banach algebra and $\phi: \mathcal{A} \to \mathbb{C}$ be a map satisfying 
	\begin{enumerate}[label=(\roman*),nosep]
		\item $\phi(x)\in \sigma(x)$,
		\item $\phi(xy)=\phi(x)\phi(y)$
	\end{enumerate}
  for all $x,y \in \mathcal{A}$. Then there exists a unique character $\psi_\phi$ on $\mathcal{A}$ such that $\phi(x)=\psi_\phi(x)$ for all $x\in G_1(\mathcal{A})$, where $G_1(\mathcal{A})$ is the connected component of $G(\mathcal{A})$ containing the identity. 
\end{thm}

In the same paper \cite{maouche1996formes}, Maouche gave a concrete example to establish that even on the nicest Banach algebra, namely $C[0,1]$, one can't guarantee the linearity of $\phi$ on the entire Banach algebra. It is noteworthy that Maouche's function was not continuous; hence, any hope of proving that $\phi$ itself is linear, therefore, a character would have to include the continuity of $\phi$ as an assumption. 

Recently, Tou$\acute{\rm r}$e and Brits, along with their other collaborators, have extensively studied multiplicative versions of the GKZ and its generalization–the KS theorem–on $C^*$-algebras; for their specific results refer to \cite{brits2021multiplicative, brits2023multiplicative, toure2017multiplicative, toure2018some}. In \cite{brits2021multiplicative}, they showed that a multiplicative continuous complex-valued function $F$ on a $C^*$-algebra $A$ satisfying $F(a)\in \sigma(a)$ must be linear. However, the result for a general Banach algebra stays where Maouche left it.  

We study an analogous result for the Hardy spaces. Before we give the statement, we would like to note that, in light of Theorem \ref{MR1}, a non-zero continuous multiplicative linear functional on a Hardy space must be a point evaluation. 

 \begin{thm} For $0<p< \infty$, let $F:H^p(\bb D) \to \mathbb{C} $ be a continuous function such that $F(f)\in Im(f)$ and $F(fg)=F(f)F(g)$, whenever $f,g\in H^p(\bb D)$ with $fg\in H^p(\bb D)$. Then there exists a point $w\in \bb D$ such that $F(f)=f(w)$ for every nowhere-vanishing $f\in H^p(\bb D)$. 
\end{thm}
\begin{proof} We consider the restriction of the map $F$ to the disc algebra $A(\bb D)$ (a subset of $H^p(\bb D)$ ). The restriction of the map $F$ to $A(\bb D)$ (equipped with the sup-norm) satisfies the hypotheses of Theorem \ref{Mao}. As a result, there exists a character $\psi_F$ on $A(\bb D)$ such that $F(f)=\psi_F(f)$ for every $f\in G_1(A(\bb D)),$ the connected component of the set of invertible functions in $A(\bb D)$ containing the constant function $1$. But, $G_1(A(\bb D))$ equals the set of all invertible functions in $A(\bb D)$. Indeed, a simple adaptation of the arguments used in \cite[Lemma 3.5.14]{murphy1990c} can be used to establish this fact. Therefore, we conclude that $F(f)=\psi_F(f)$ for each invertible function in $A(\bb D).$  

Furthermore, the characters on the disc algebra are point evaluations, which implies that there exists a point $w\in \overline{\bb D}$ such that 
$F(f)=\psi_F(f)=f(w)$ for every invertible $f\in A(\bb D).$ Now, $g(z)=z+2$ is an invertible function in $A(\bb D);$ therefore, $F(g)=\psi_F(g)=w+2.$ On the other hand, by the hypotheses, $F(g)=a+2$ for some $a\in \bb D$. Thus, we conclude that $w=a\in \bb D.$  

Finally, let $f\in H^p(\bb D) $ with no zero in $\bb D$. For $0<r<1$, let us define $f_r(z)=f(rz), \ z\in \overline{\bb{D}}$. Then each $f_r$ is in $H^p(\bb D)$ and $f_r\to f$ in $H^p(\bb D)$ as $r\to 1.$ 
Since $f$ is analytic on $\bb D$ and has no zero in $\bb D$, each $f_r$ is an invertible function in $A(\bb D)$. Hence, 
$F(f_r)=\psi_F(f_r)=f_r(w)$ for every $r.$ Lastly, the continuity of $F$ yields $F(f)=\lim\limits_{r\to {1^-}}F(f_r)=\lim\limits_{r\to {1^-}}f_r(w)=f(w).$ This completes the proof. 
\end{proof}

\subsection*{Acknowledgements} We thank the Mathematical Sciences Foundation, Delhi, for the support and facilities needed to complete the present work. The first named author thanks the University Grants Commission(UGC), India, for the support, and the second named author thanks Shiv Nadar Institution of Eminence for partially supporting this research.

\bibliographystyle{plain}

\bibliography{References1}
\end{document}